\documentclass[reqno,11pt]{amsart}
\usepackage[foot]{amsaddr}
\usepackage{bbold}
\usepackage{charter}
\usepackage[margin=1in]{geometry}
\usepackage[colorlinks=true,linktoc=all,linkcolor=blue,citecolor=blue]{hyperref}

\theoremstyle{plain}
\newtheorem{theorem}{Theorem}[section]
\newtheorem*{theorem*}{Theorem}

\newtheorem{lemma}[theorem]{Lemma}

\theoremstyle{definition}

\newtheorem{definition}[theorem]{Definition}
\newtheorem*{definition*}{Definition}

\numberwithin{equation}{section}
\everymath{\displaystyle}
\allowdisplaybreaks

\newcommand{\C}{\mathbb{C}}
\newcommand{\D}{\mathbb{D}}
\newcommand{\N}{\mathbb{N}}
\newcommand{\Lip}{\mathcal{L}}
\newcommand{\LipW}{\Lip_{\normalfont{\textbf{w}}}}
\newcommand{\Linf}{L^\infty}
\newcommand{\LW}{L^\infty_{\normalfont{\textbf{w}}}}
\newcommand{\Lipk}{\Lip^{(k)}}
\newcommand{\Lipm}{\Lip^{(m)}}
\newcommand{\Lipn}{\Lip^{(n)}}
\newcommand{\bigchi}{\mbox{\Large$\chi$}}

\title[Multiplication Operators Between I.L.L. Spaces]{Multiplication Operators Between Iterated Logarithmic Lipschitz Spaces\\ of a Tree}

\author{Robert F.~Allen\textsuperscript{1}, Flavia Colonna\textsuperscript{2}, and Andrew Prudhom\textsuperscript{3}}
\address{\textsuperscript{1}Department of Mathematics and Statistics, University of Wisconsin-La Crosse}
\address{\textsuperscript{2}Department of Mathematical Science, George Mason UniversityA}
\address{\textsuperscript{3}Department of Mathematics, University of North Carolina at Chapel Hill}

\email{rallen@@uwlax.edu, fcolonna@gmu.edu, prudhom@live.unc.edu}


\date{}

\dedicatory{In memory of Maurice Heins}

\begin{document}
\maketitle

\begin{abstract}
In this article, we characterize the bounded and the compact multiplication operators between distinct iterated logarithmic Lipschitz spaces, and between the Lipschitz space and an iterated logarithmic Lipschitz space of an infinite tree.  In addition, we provide operator norm estimates and show that there are no isometries among such operators.  
\end{abstract}

\section{Introduction}
Let $X$ and $Y$ be Banach spaces of complex-valued functions defined on a set $\Omega$.  For a complex-valued function $\psi$ on $\Omega$ we define the \textit{multiplication operator with symbol $\psi$} to be the linear operator $M_\psi f = \psi f$ for all $f \in X$.  A primary objective in the study of operators with symbol is to relate the function-theoretic properties of the symbol to the properties of the operator.  

In recent years, the study of such operators on spaces of functions defined on discrete structures has been conducted.  In particular, the first two authors, among others, have studied operators on spaces of functions defined on infinite rooted trees.  These spaces, under certain restrictions, are often considered as discrete analogs to  classical spaces of analytic functions defined on the open unit disk $\D$.  A connection between functions on a homogeneous tree and analytic functions on $\D$ was made explicit through a particular embedding of the tree into $\D$ in \cite{CohenColonna:94}.

Motivated by the wide interest in the study of linear operators on the Bloch space of analytic functions on $\D$, whose elements are characterized by a Lipschitz condition under an appropriate choice of metrics (see \cite{Colonna:89}), the Lipschitz space $\Lip$ on a tree $T$ was introduced in \cite{ColonnaEasley:10}.  The study of multiplication and composition operators on $\Lip$ conducted in \cite{ColonnaEasley:10} and \cite{AllenColonnaEasley:14} led, in \cite{AllenColonnaEasley:13}, to the study of the weighted Lipschitz space $\LipW$ and of the multiplication operators acting on $\LipW$. For related work on such operators see \cite{AllenColonnaEasley:11}. Due to the analogy between the Bloch space and the Lipschitz space in their respective environments, the weighted Lipschitz space can be viewed as the discrete analog of the weighted Bloch space studied in \cite{Yoneda:02}.

The process that gave rise to $\LipW$ was used to define the iterated logarithmic Lipschitz spaces, $\Lipk$ for $k \in \N$, of which $\LipW = \Lip^{(1)}$.  The multiplication operators on $\Lipk$ were studied in \cite{AllenColonnaEasley:12}.  These spaces can be viewed as discrete versions of the logarithmic Bloch spaces, on which the weighted composition operators were studied in \cite{HosokawaDieu:09}. 

Research on the multiplication operators acting between the space $L^\infty$ of bounded functions on a tree equipped with the supremum norm $\|\cdot\|_\infty$, and $\Lip$ was carried out in \cite{ColonnaEasley:12}.  In \cite{AllenCraig:2015} and \cite{AllenPons:2016},  the first author with Craig and Pons, respectively, studied multiplication and composition operators on the weighted Banach space $L_\mu^\infty$ of functions on a tree for a given positive weight $\mu$.  In \cite{MuthukumarPonnusamyI} and \cite{MuthukumarPonnusamyII}, the authors developed a discrete version of the Hardy spaces on homogeneous trees and studied the multiplication and composition operators on such spaces.

In this paper, we expand the research carried out in \cite{AllenColonnaEasley:12} by focusing on the multiplication operators acting either between distinct iterated logarithmic Lipschitz spaces, or between the Lipschitz space and an iterated logarithmic Lipschitz space. The techniques used also provide improvements on known results in some cases.

\subsection{Preliminary Definitions and Notation}
By a \textit{tree} $T$ we mean a locally finite, connected, and simply-connected graph, which, as a set, we identify with the collection of its vertices.  By a \textit{function on a tree} we mean a complex-valued function on the set of its vertices.

Two vertices $v$ and $w$ are called \textit{neighbors} if there is an edge $[v,w]$
connecting them, and we use the notation $v\sim w$. A vertex is called \textit{terminal} if it has a unique neighbor. A \textit{path} is a sequence of vertices $[v_0,v_1,\dots]$ such that $v_k\sim v_{k+1}$ and  $v_{k-1}\ne v_{k+1}$, for all $k$.  Define the \textit{length} of a finite path $[v=v_0,v_1,\dots,w=v_n]$ to be the number  of edges $n$ connecting $v$ to $w$. The \textit{distance} between vertices $v$ and $w$ is the length $d(v,w)$ of the unique path connecting $v$ to $w$.

Given a tree $T$ rooted at $o$, the \textit{length} of a vertex $v$ is defined as $|v|=d(o,v)$. For a vertex $v\in T$, a vertex $w$ is called a \textit{descendant} of $v$  if $v$ lies in the path from $o$ to $w$. The vertex $v$ is then called an \textit{ancestor} of $w$.  For $v \in T$ with $v \neq o$, we denote by $v^-$ the unique neighbor which is an ancestor of $v$.  The vertex $v$ is called a \textit{child} of $v^-$. 
For $v\in T$, the set $S_v$ consisting of $v$ and all its descendants is called the \textit{sector} determined by $v$. The set $T\backslash\{o\}$ will be denoted by $T^*$. 

In this paper, we shall assume the tree $T$ to be without terminal vertices (and hence infinite), and rooted at a vertex $o$.  In addition, $k,m,$ and $n$ are natural numbers, and $\psi$ a fixed function on $T$.  We define the \textit{discrete derivative} of a function $f$ on $T$ as $$f'(v) = \begin{cases}f(v)-f(v^-) & \text{if } v \neq o,\\0 & \text{if } v = o.\end{cases}$$

\section{Lipschitz-Type Spaces}\label{Section:Family}
In \cite{ColonnaEasley:10}, the \textit{Lipschitz space} was defined as the set of functions on $T$ which are Lipschitz as maps between the metric space $(T,d)$ and the Euclidean space $\C$. It was shown that a function $f\in \Lip$ if and only if $f' \in \Linf$, and $\|f'\|_\infty$ is precisely the Lipschitz number of $f$.  It was further shown that $\Lip$ is a Banach space under the norm 
$$\|f\|_\Lip = |f(o)| + \|f'\|_\infty.$$  The following result provides a bound on point-evaluation in the Lipschitz space.

\begin{lemma}\label{lip_bound}\cite[Lemma 3.4(a)]{ColonnaEasley:10} 
	If $f \in \Lip$ and $v \in T$, then $$|f(v)| \leq |f(o)| + |v|\|f'\|_\infty.$$  In particular, if $\|f\|_\Lip \leq 1$, then $|f(v)| \leq |v|$ for each $v \in T^*$.
\end{lemma}

The bounded multiplication operators on $\Lip$ were characterized as follows.
\begin{theorem}\cite[Theorem 3.6]{ColonnaEasley:10} The operator $M_\psi$ is bounded on $\Lip$ if and only if $\psi \in L^\infty$ and $\psi' \in \LW$, where {\bf{w}} is the weight defined by the length, that is, {\bf{w}}$(v)=|v|$.

\end{theorem}
This result motivated the study in \cite{AllenColonnaEasley:13} of the \textit{weighted Lipschitz space} $\LipW$ defined as
the set of functions $f$ on $T$ such that $f' \in \LW$.  Thus, the bounded functions in $\LipW$ are precisely those that induce bounded multiplication operators on $\Lip$.  
The bounded multiplication operators on $\LipW$ were characterized in the following result.
\begin{theorem}\cite[Theorem 4.1]{AllenColonnaEasley:13} The operator $M_\psi$ is bounded on $\LipW$ if and only if $\psi \in L^\infty$ and 
$\psi' \in L^\infty_\mu$, where $\mu(v)=|v|\log|v|$ for $v\in T^*$.
\end{theorem}
In fact, this process can be continued, thus creating the iterated logarithmic Lipschitz spaces, defined in \cite{AllenColonnaEasley:12}.  For $x \geq 1$ define the recursive sequence $\ell_j(x)$ by 
	$$\ell_{j}(x)=\begin{cases}
	x& \text{if } j=0,\\1+\log x& \text{if } j=1,\\1+\log \ell_{j-1}(x)& \text{if } j\geq 2.\end{cases}$$
In addition, the sequence $\Lambda_k(x)$ is defined as
$$\Lambda_k(x) = \begin{cases}
	1&\text{if } k=0,\\\displaystyle\prod_{j=0}^{k-1} \ell_j(x)& \text{if } k \geq 1.
	\end{cases}$$

\begin{definition}\label{Lipit} Let $T$ be a tree rooted at $o$.  For non-negative integer $k$, the {\textit{iterated logarithmic Lipschitz space}} $\Lipk$ is the set of functions $f$ on $T$ satisfying the condition $$\sup_{v\in T^*}|f'(v)|\Lambda_k(|v|)<\infty.$$  For $f \in \Lipk$, define $$\|f\|_k = |f(o)| + \sup_{v \in T^*} |f'(v)|\Lambda_k(|v|).$$
\end{definition}

Notice that $\Lip = \Lip^{(0)}$ and $\LipW = \Lip^{(1)}.$  In \cite[Proposition 3.3]{ColonnaEasley:10} and \cite[Proposition 2.2]{AllenColonnaEasley:12} it was shown that $\Lipk$ is a functional Banach space under the norm $\|f\|_k$. Moreover, the point-evaluation functionals were proven to be bounded.  With Lemma \ref{lip_bound}, we obtain the following result.

\begin{lemma}\label{lipk_bound} 
If $f \in \Lipk$ and $v \in T^*$, then $$|f(v)| \leq \begin{cases}(1+|v|)\|f\|_k&\text{if } k=0,\\\ell_k(|v|)\|f\|_k&\text{if } k\geq 1.\end{cases}$$  Moreover, if $\|f\|_k \leq 1$, then $|f(v)| \leq \ell_k(|v|)\|f\|_k$ for all $k \geq 0$.
\end{lemma}

The iterated logarithmic Lipschitz spaces are continuously embedded in one another,  as shown in the next result.

\begin{theorem} If $m \leq n$, then $\Lip^{(n)} \subseteq \Lip^{(m)}$ and $\|f\|_{m}\le \|f\|_n$ for all $f\in \Lip^{(n)}$.
\end{theorem}

\begin{proof} It suffices to show that each space $\Lip^{(k)}$ is contained in the immediate predecessor space $\Lip^{(k-1)}$.  Let $f$ be a function on $T$.  Since $$|f'(v)| \leq |v||f'(v)| \leq |v|(1+\log |v|)|f'(v)|$$ for all $v \in T^*$, we have $$\sup_{v \in T^*} |f'(v)|\Lambda_0(|v|) \leq \sup_{v \in T^*} |f'(v)|\Lambda_1(|v|) \leq \sup_{v \in T^*} |f'(v)|\Lambda_2(|v|).$$  Thus $\|f\|_0 \leq \|f\|_1 \leq \|f\|_2$, and so $\Lip^{(2)} \subseteq \Lip^{(1)} \subseteq \Lip^{(0)}$. 
	
Next, assume $k > 2$.  Then
$$\begin{aligned}
\sup_{v \in T^*} |f'(v)|\Lambda_k(|v|) &= \sup_{v \in T^*} |f'(v)|\ell_{k-1}(|v|)\Lambda_{k-1}(|v|)\\
&= \sup_{v \in T^*} |f'(v)|(1+\log \ell_{k-2}(|v|))\Lambda_{k-1}(|v|)\\
&\geq \sup_{v \in T^*} |f'(v)|\Lambda_{k-1}(|v|).
\end{aligned}$$ Thus, we have $\|f\|_{k-1} \leq \|f\|_k$, which implies $\Lip^{(k)} \subseteq \Lip^{(k-1)}$.
\end{proof}

The following result is inspired by \cite[Lemma 2.10]{Tjani:96}, where it was proved for Banach spaces of analytic functions on $\D$.  

\begin{lemma}\label{Lemma:compactness}\cite[Lemma 2.5]{AllenCraig:2015} Let $X$  and  $Y$ be Banach spaces of functions on $T$.  Suppose that
	\begin{enumerate}
		\item the point-evaluation functionals of $X$ are bounded;
		\item the closed unit ball of $X$ is a compact subset of $X$ in the topology of pointwise convergence; 
		\item $T:X \to Y$ is bounded when $X$ and $Y$ are given the topology of pointwise convergence.
	\end{enumerate}  Then $T$ is a compact operator if and only if given a bounded sequence $(f_n)$ in $X$ converging to 0 pointwise, the sequence $(T f_n)$ converges to zero in the norm of $Y$.
\end{lemma}

\section{Boundedness and Operator Norm}\label{Section:Boundedness}
In this section, we characterize the bounded multiplication operators between distinct iterated logarithmic Lipschitz spaces, and, in particular, between the Lipschitz and weighted Lipschitz spaces.  In addition, we provide estimates on the operator norm.  

The boundedness of $M_\psi$ of $\Lipk$ was proven for $k=0$ in \cite[Theorem 3.6]{ColonnaEasley:10} and for $k \geq 1$ in \cite[Theorem 3.1]{AllenColonnaEasley:12}, showing that the bounded functions in $\Lipk$ induce bounded multiplication operators on $\Lip^{(k+1)}$.
\begin{theorem}\label{Theorem:BoundedMultOnLipm}
Let $k$ be a non-negative integer.  Then the operator $M_\psi$ is bounded on $\Lipk$ if and only if $\psi \in L^\infty \cap \Lip^{(k+1)}$ for every non-negative $k$. Furthermore, the following estimates hold:
	$$\max\{\|\psi\|_k,\|\psi\|_\infty\} \leq \|M_\psi\| \leq \|\psi\|_\infty + \sup_{v \in T^*} |\psi'(v)|\Lambda_{k+1}(|v|).$$
\end{theorem}

To characterize the boundedness of $M_\psi$ acting between distinct iterated logarithmic spaces we define the following two quantities.  For $\psi$ a function on $T$, and $m$ and $n$ distinct non-negative integers, define 
$$\begin{aligned}
\mu_{\psi,m,n} &= \sup_{v \in T^*} |\psi'(v)|\ell_m(|v|)\Lambda_n(|v|),\\
\nu_{\psi,m,n} &= \sup_{v \in T^*} \frac{|\psi (v^-)|\Lambda_n(|v|)}{\Lambda_m(|v|)}.
\end{aligned}$$

\begin{theorem}\label{Theorem:boundedness}
Let $m$ and $n$ be distinct non-negative integers.  Then the operator $M_\psi:\Lipm \to \Lipn$ is bounded if and only if $\mu_{\psi,m,n}$ and $\nu_{\psi,m,n}$ are both finite.  Furthermore, under these conditions, $$\|\psi\|_n \leq \|M_\psi\| \leq |\psi(o)| + \mu_{\psi,m,n} + \nu_{\psi,m,n}.$$
\end{theorem}

\begin{proof} Suppose $\mu_{\psi,m,n}$ and $\nu_{\psi,m,n}$ are both finite.  Let $f \in \Lipm$ such that $\|f\|_m \leq 1$.  By Lemma \ref{lipk_bound}, for $v \in T^*$
$$\begin{aligned}
|(\psi f)'(v)|\Lambda_n(|v|) &\leq |\psi'(v)||f(v)|\Lambda_n(|v|) + |\psi(v^-)||f'(v)|\Lambda_n(|v|)\\
&\leq |\psi'(v)|\ell_m(|v|)\Lambda_n(|v|)\|f\|_m + \frac{|\psi(v^-)||f'(v)|\Lambda_n(|v|)\Lambda_m(|v|)}{\Lambda_m(|v|)}\\
&\leq |\psi'(v)|\ell_m(|v|)\Lambda_n(|v|)\|f\|_m + \frac{|\psi(v^-)|\Lambda_n(|v|)}{\Lambda_m(|v|)}\|f\|_m\\
&\leq \mu_{\psi,m,n} + \nu_{\psi,m,n}.
\end{aligned}$$  So $\psi f \in \Lipn$ with $$\|M_\psi f\|_n = |\psi(o)f(o)| + \sup_{v \in T^*} |(\psi f)'(v)|\Lambda_n(|v|) \leq |\psi(o)| + \mu_{\psi,m,n}+\nu_{\psi,m,n}.$$  Thus $M_\psi : \Lipm \to \Lipn$ is bounded, and the upper bound of the operator norm is established.

Next, suppose $M_\psi : \Lipm \to \Lipn$ is bounded.  Then $\psi = M_\psi 1$ is an element of $\Lipn$.  So $\|\psi\|_n \leq \|M_\psi\|$, establishing the lower bound of the operator norm.  

To prove the finiteness of $\nu_{\psi,m,n}$, fix a vertex $v \in T$ and define for all $w \in T$ $$f_v(w) = \frac{\bigchi_v(w)}{\Lambda_m(|v|+1)},$$ where $\bigchi_v$ denotes the characteristic function of $\{v\}$.  Thus $$|f'_v(w)| = \begin{cases}\displaystyle\frac{1}{\Lambda_m(|v|+1)} & \text{if } w = v \text{ or } w^- = v,\\0 & \text{otherwise}.\end{cases}$$  From this, we obtain $f_v \in \Lipm$ since $\sup_{w \in T^*} |f'_v(w)|\Lambda_m(|w|) = 1$.  Since $M_\psi$ is bounded, 
$$\frac{|\psi(v)|\Lambda_n(|v|+1)}{\Lambda_m(|v|+1)} = \sup_{w \in T^*} |(\psi f_v)'(w)|\Lambda_n(|w|) \leq \|M_\psi\|.$$  Thus, since $v$ is arbitrary, we have 
\begin{equation}\label{bounded_lower_bound}\nu_{\psi,m,n} = \sup_{v \in T^*} \frac{|\psi(v^-)|\Lambda_n(|v|)}{\Lambda_m(|v|)} \leq \|M_\psi\| < \infty.\end{equation} 

Finally, to prove the finiteness of $\mu_{\psi,m,n}$, define 
$$g(v) = \begin{cases}
0&\text{if } v=o,\\
\ell_m(|v|)&\text{if } v \neq o.
\end{cases}\label{f=lk}$$
The function $g$ is an element of $\Lipm$ for all $m \geq 0$ since $\|g\|_0 = 1$ and for $m \geq 1$, $\|g\|_m \le 2\prod_{j=1}^m \ell_j(2)$, as shown in \cite[proof of Theorem 3.1]{AllenColonnaEasley:12}. From the boundedness of $M_\psi$, we have $$\sup_{v \in T^*} |(\psi g)'(v)|\Lambda_n(|v|) \leq \|M_\psi\|\|g\|_{m}.$$  On the other hand, for $v \in T^*$ we obtain 
$$\begin{aligned} 
|\psi'(v)|\ell_m(|v|)\Lambda_n(|v|) &= |\psi'(v)||g(v)|\Lambda_n(|v|)\\ 
&\leq |(\psi g)'(v)|\Lambda_n(|v|) + |\psi(v^-)||g'(v)|\Lambda_n(|v|)\\
&\leq \|M_\psi\|\|g\|_m + \frac{|\psi(v^-)|\Lambda_n(|v|)}{\Lambda_m(|v|)}\|g\|_m.
\end{aligned}$$  Taking the supremum over all $v \in T^*$, we obtain $$\mu_{\psi,m,n} \leq \big(\|M_\psi\| + \nu_{\psi,m,n}\big)\|g\|_m.$$  Thus $\mu_{\psi,m,n}$ is finite.
\end{proof}

\section{Compactness}\label{Section:Compactness}
In this section, we characterize the compact multiplication operators among those studied in Section \ref{Section:Boundedness}.  As was the case for boundedness of $M_\psi$ acting from $\Lipm$ to $\Lipn$, the characterizing quantities for compactness are precisely $|\psi'(v)|\ell_m(|v|)\Lambda_n(|v|)$ and $\frac{|\psi(v^-)\Lambda_n(|v|)}{\Lambda_m(|v|)}$.  The characterization for the bounded multiplication operators used the growth behavior, or ``Big-Oh", of these quantities.  As is seen in other situations, the compactness is characterized by the asymptotic behavior, or ``little-oh", of these quantities.  To this end, define for $v \in T^*$ 

$$\begin{aligned}
\mu_{\psi,m,n}(v) &= |\psi'(v)|\ell_m(|v|)\Lambda_n(|v|),\\
\nu_{\psi,m,n}(v) &= \frac{|\psi(v^-)|\Lambda_n(|v|)}{\Lambda_m(|v|)}.
\end{aligned}$$

\noindent Observe the reuse of notation from Section \ref{Section:Boundedness}, since $\mu_{\psi,m,n} = \sup_{v \in T^*} \mu_{\psi,m,n}(v)$ and likewise for $\nu_{\psi,m,n}$.

\begin{theorem}
Let $m$ and $n$ be distinct non-negative integers.  Then the bounded operator $M_\psi:\Lipm \to \Lipn$ is compact if and only if $$\lim_{|v|\to\infty} \mu_{\psi,m,n}(v) = 0 \text{ and } \lim_{|v|\to\infty} \nu_{\psi,m,n}(v) = 0.$$
\end{theorem}

\begin{proof}
First, suppose $\mu_{\psi,m,n}(v) \to 0$ and $\nu_{\psi,m,n}(v) \to 0$ as $|v| \to \infty$.  Let $(f_k)$ be a bounded sequence in $\Lipm$ which converges to 0 pointwise, and define $s = \sup_{k \in \N} \|f_k\|_m$.  By Lemma \ref{Lemma:compactness}, it suffices to show that $\|M_\psi f_k\|_n \to 0$ as $k \to \infty$.

Fix $\varepsilon > 0$.  Since $f_k \to 0$ pointwise, $|f_k(o)| < \frac{\varepsilon}{3s}$ for all $k$ sufficiently large.  By assumption, there exists $M > 0$ such that $$|\psi'(v)|\ell_m(|v|)\Lambda_n(|v|) < \frac{\varepsilon}{3s}$$ and $$\frac{|\psi(v^-)|\Lambda_n(|v|)}{\Lambda_m(|v|)} < \frac{\varepsilon}{3s}$$ for all $|v| \geq M$.

Applying Lemma \ref{lipk_bound}, and normalizing $f_k$ if $\|f_k\|_m > 1$, for all $|v| \geq M$
$$\begin{aligned}
|(\psi f_k)'(v)|\Lambda_n(|v|) &\leq |\psi'(v)||f_k(v)|\Lambda_n(|v|) + |\psi(v^-)||f'_k(v)|\Lambda_n(|v|)\\
&\leq |\psi'(v)|\ell_m(|v|)\Lambda_n(|v|)\|f_k\|_m + \frac{|\psi(v^-)|\Lambda_n(|v|)}{\Lambda_m(|v|)}\|f_k\|_m\\
&< \frac{2\varepsilon}{3}.
\end{aligned}$$

\noindent Since $(f_k)$ converges to 0 on $\left\{v \in T : |v| \leq M\right\}$, so does the sequence $\left(|(\psi f_k)'(v)|\Lambda_n(|v|)\right)$.  So for sufficiently large $k$, and all $v \in T^*$, we have $|(\psi f_k)'(v)|\Lambda_n(|v|) < \frac{2\varepsilon}{3}$.  Thus $$\|M_\psi f_k\|_n = |\psi(o)||f_k(o)| + \sup_{v \in T^*} |(\psi f_k)'(v)|\Lambda_n(|v|) < \varepsilon$$ for all $k$ sufficiently large.  Thus $\|M_\psi f_k\|_n \to 0$ as $k \to \infty$, proving the compactness of $M_\psi$.  

Conversely, suppose $M_\psi:\Lipm \to \Lipn$ is compact.  Let $(v_k)$ be a sequence in $T$ such that $1 \leq |v_k| \to \infty$.  For $k \in \N\cup\{0\}$ and $w \in T$, define $$g_k(w) = \frac{\bigchi_{v_{k}^-}(w)}{\Lambda_m(|v_k|)}.$$  Observe that $g_0(w) = \frac{\bigchi_{v_{0}^-}(w)}{\Lambda_m(|v_0|)}$, thus $g_0 \in \Lip^{0}$ with $\|g_0\|_0 \leq 1$.  For $k \geq 1$, $g_k \in \Lipm$ with $\|g_k\|_m = 1$, as shown in the proof of Theorem \ref{Theorem:boundedness} (with a minor modification).  Moreover, $(g_k)$ converges to 0 pointwise.  Since $M_\psi$ is compact, by Lemma \ref{Lemma:compactness}, it follows that $\|M_\psi g_k\|_n \to 0$ as $k \to \infty$.  Observing that $$|(\psi g_k)'(v)| = \frac{|\psi(v_k^-)|}{\Lambda_m(|v_k|)}$$ we obtain
$$\frac{|\psi(v_k^-)|\Lambda_n(|v_k|)}{\Lambda_m(|v_k|)} = |(\psi g_k)'(v_k)|\Lambda_n(|v_k|) \leq \|M_\psi g_k\|_n \to 0.$$  Thus $\nu_{\psi,m,n}(v) \to 0$ as $|v| \to \infty$.

Lastly, corresponding to the above sequence $(v_k)$, where without loss of generality we now assume $|v_k| > 3$, define
$$h_k(w) = \begin{cases}
0 &\text{if } |w| = 0 \text{ or } 1,\\
\frac{\ell_m(|w|)^2}{\ell_m(|v_k|)} &\text{if } 2 \leq |w| < |v_k|-1,\\
\ell_m(|v_k|) &\text{if } |w| \geq |v_k|-1.
\end{cases}$$  From \cite[Theorem 6.2]{AllenColonnaEasley:12} we have $h_k \in \Lipm$ and $\|h_k\|_m$ is bounded for all $m$.  Moreover, $(h_k)$ converges to 0 pointwise.  Therefore, by the compactness of $M_\psi$, we have
$$|\psi'(v_k)|\ell_m(|v_k|)\Lambda_n(|v_k|) = |(\psi h_k)'(v_k)|\Lambda_n(|v_k|) \leq \|M_\psi h_k\|_n \to 0,$$ as $k \to \infty$.  Thus $\mu_{\psi,m,n}(v) \to 0$ as $|v| \to \infty$, completing the proof. 
\end{proof}

\section{Isometries}\label{Section:Isomtries}
From \cite[Theorem 9.1]{ColonnaEasley:10} and \cite[Theorem 5.1]{AllenColonnaEasley:12}, the only isometric multiplication operators acting on $\Lipk$ are induced by unimodular constant functions, for all $k \geq 0$.  In this section, we show there are no isometries among the multiplication operators acting between distinct iterated logarithmic spaces.  

\begin{theorem} Let $m$ and $n$ be distinct non-negative integers.  Then $M_\psi:\Lipm \to \Lipn$ is not an isometry.
\end{theorem}

\begin{proof}
Assume $M_\psi$ is an isometry.  Taking as test functions the constant function 1 and $g(v) = \frac12\bigchi_o(v)$, we obtain
$$\|\psi\|_n = \|M_\psi 1\|_n = \|1\|_m = 1 = \|g\|_m = \|M_\psi g\|_n = \frac12\|\psi \bigchi_o\|_n = |\psi(o)|.$$  So $$\sup_{v \in T^*} |\psi'(v)|\Lambda_n(|v|) = \|\psi\|_n - |\psi(o)| = 0.$$  Since $\Lambda_n(|v|) \geq 1$ for all $v \in T^*$, it follows that $\psi'$ is identically zero.  Thus $\psi$ is a unimodular constant function.

Fix vertex $w$ in $T$ such that $|w| \geq 2$.   Then $$\Lambda_m(|w|) = \|\bigchi_w\|_m = \|M_\psi \bigchi_w\|_n = |\psi(w)|\Lambda_n(|w|) = \Lambda_n(|w|).$$  Observe that the sequence $\Lambda_k(|w|)$ is strinctly increasing in $k$ since $|w| \geq 2$.  So it can not be the case that $\Lambda_m(|w|) = \Lambda_n(|w|)$.  Therefore, $M_\psi$ is not an isometry.
\end{proof}

\bibliographystyle{amsplain}
\bibliography{references.bib}

\providecommand{\bysame}{\leavevmode\hbox to3em{\hrulefill}\thinspace}
\providecommand{\MR}{\relax\ifhmode\unskip\space\fi MR }
\providecommand{\MRhref}[2]{%
  \href{http://www.ams.org/mathscinet-getitem?mr=#1}{#2}
}
\providecommand{\href}[2]{#2}
\begin{thebibliography}{10}

\bibitem{AllenColonnaEasley:11}
Robert~F. Allen, Flavia Colonna, and Glenn~R. Easley, \emph{Multiplication
  operators between {L}ipschitz-type spaces on a tree}, Int. J. Math. Math.
  Sci. (2011), Art. ID 472495, 36. \MR{2799856}

\bibitem{AllenColonnaEasley:12}
\bysame, \emph{Multiplication operators on the iterated logarithmic {L}ipschitz
  spaces of a tree}, Mediterr. J. Math. \textbf{9} (2012), no.~4, 575--600.
  \MR{2991152}

\bibitem{AllenColonnaEasley:13}
\bysame, \emph{Multiplication operators on the weighted {L}ipschitz space of a
  tree}, J. Operator Theory \textbf{69} (2013), no.~1, 209--231. \MR{3029495}

\bibitem{AllenColonnaEasley:14}
\bysame, \emph{Composition operators on the {L}ipschitz space of a tree},
  Mediterr. J. Math. \textbf{11} (2014), no.~1, 97--108. \MR{3160615}

\bibitem{AllenCraig:2015}
Robert~F. Allen and Isaac~M. Craig, \emph{Multiplication operators on weighted
  {B}anach spaces of a tree}, Bull. Korean Math. Soc. \textbf{54} (2017),
  no.~3, 747--761. \MR{3659146}

\bibitem{AllenPons:2016}
Robert~F. Allen and Matthew~A. Pons, \emph{Composition operators on weighted
  {B}anach spaces of a tree}, Bull. Malays. Math. Sci. Soc. \textbf{41} (2018),
  no.~4, 1805--1818. \MR{3854493}

\bibitem{CohenColonna:94}
Joel~M. Cohen and Flavia Colonna, \emph{Embeddings of trees in the hyperbolic
  disk}, Complex Variables Theory Appl. \textbf{24} (1994), no.~3-4, 311--335.
  \MR{1270321}

\bibitem{Colonna:89}
Flavia Colonna, \emph{Bloch and normal functions and their relation}, Rend.
  Circ. Mat. Palermo (2) \textbf{38} (1989), no.~2, 161--180. \MR{1029707}

\bibitem{ColonnaEasley:12}
Flavia Colonna and Glenn Easley, \emph{Multiplication operators between the
  {L}ipschitz space and the space of bounded functions on a tree}, Mediterr. J.
  Math. \textbf{9} (2012), no.~3, 423--438. \MR{2954500}

\bibitem{ColonnaEasley:10}
Flavia Colonna and Glenn~R. Easley, \emph{Multiplication operators on the
  {L}ipschitz space of a tree}, Integral Equations Operator Theory \textbf{68}
  (2010), no.~3, 391--411. \MR{2735443}

\bibitem{HosokawaDieu:09}
Takuya Hosokawa and Nguyen~Quang Dieu, \emph{Weighted composition operators on
  the logarithmic {B}loch spaces with iterated weights}, Nihonkai Math. J.
  \textbf{20} (2009), no.~1, 57--72. \MR{2572800}

\bibitem{MuthukumarPonnusamyII}
Perumal Muthukumar and Saminathan Ponnusamy, \emph{Composition operators on the
  discrete analogue of generalized {H}ardy space on homogenous trees}, Bull.
  Malays. Math. Sci. Soc. \textbf{40} (2017), no.~4, 1801--1815. \MR{3712587}

\bibitem{MuthukumarPonnusamyI}
\bysame, \emph{Discrete analogue of generalized {H}ardy spaces and
  multiplication operators on homogenous trees}, Anal. Math. Phys. \textbf{7}
  (2017), no.~3, 267--283. \MR{3683009}

\bibitem{Tjani:96}
Maria Tjani, \emph{Compact composition operators on some {M}oebius invariant
  {B}anach spaces}, ProQuest LLC, Ann Arbor, MI, 1996, Thesis (Ph.D.)--Michigan
  State University. \MR{2695395}

\bibitem{Yoneda:02}
Rikio Yoneda, \emph{The composition operators on weighted {B}loch space}, Arch.
  Math. (Basel) \textbf{78} (2002), no.~4, 310--317. \MR{1895504}

\end{thebibliography}
\end{document}